\documentclass[11pt]{article}

\usepackage{amsmath,amsthm,verbatim,amssymb,amsfonts,amscd, graphicx,color, framed,enumerate}

\newcommand{\R}{\mathbb R}
\newcommand{\Z}{\mathbb Z}

\usepackage[left=1in,right=1in,top=1in,bottom=1in]{geometry}
\usepackage[colorlinks]{hyperref}

\usepackage{tikz}
\usepackage[most]{tcolorbox}
\newtcolorbox[blend into=figures]{myfigure}[2][]{float=htb,
title={#2},#1}

\usepackage[normalem]{ulem}

\theoremstyle{plain}
\newtheorem{theorem}{Theorem}[section]

\newtheorem{lemma}{Lemma}[section]
\newtheorem{remark}{Remark}
\newtheorem{proposition}{Proposition}[section]

\theoremstyle{definition}
\newtheorem{definition}{Definition}[section]

\theoremstyle{definition}
\newtheorem{example}{Example}

\def\S{\mathcal S}
\def\C{\mathcal C}
\def\Re{\mathcal R}

\def\deg{\text{deg}}
\def\var{\text{Var}}
\def\P{\mathbb{P}}
\def\E{\mathbb{E}}
\def\P{\mathbb{P}}

\begin{document}

\title{Prevalence of deficiency zero reaction networks in an Erd\H os-R\'enyi  framework}
\author{David F.~Anderson\thanks{Department of Mathematics, University of
  Wisconsin, Madison, USA.  anderson@math.wisc.edu},
  \and
  Tung D.~Nguyen\thanks{Department of Mathematics, University of
  Wisconsin, Madison, USA.  nguyen34@math.wisc.edu.
}
}
\maketitle

\begin{abstract}
Reaction networks are commonly used within the mathematical biology and mathematical chemistry communities to model the dynamics of interacting species.  These models differ from the typical graphs found in random graph theory since their vertices are constructed from elementary building blocks, i.e., the species. 
In this paper, we consider these networks in an Erd\H os-R\'enyi framework and, under suitable assumptions, derive a threshold function for the network to have a deficiency of zero, which is a property of great interest in the  reaction network community.  Specifically,   if the number of species is denoted by $n$ and if the edge probability is denote by $p_n$, then we prove that the probability of a random binary network being deficiency zero converges to  1 if $p_n\ll r(n)$, as $n \to \infty$, and converges to 0 if  $p_n \gg r(n)$, as $n \to \infty$, where  $r(n)=\frac{1}{n^3}$.

\end{abstract}

\section{Introduction}
Reaction network models are often used to study the dynamics of the abundances of species from various branches of chemistry and biology.  Here the word ``species'' can refer to different (bio)chemical molecules or different animal species, depending on the context. These networks take the form of directed graphs in which the vertices, often termed \textit{complexes} in the domains of interest, are linear combinations of the species over the non-negative integers and the directed edges, which imply a state transition for the associated dynamical system, are termed \textit{reactions}.    See Figure \ref{figure1} for an example of a reaction network. 

To each such graph a quantity termed the \textit{deficiency} can be computed, and this quantity has been central to many classical results pertaining to the associated dynamical systems \cite{AC:non-mass,ACK:product,F1,H,H-J1,AN:non-mass,CW:CB}. 
To compute the deficiency, we first note that the vertices of a reaction network, which will be denoted by $y$ and/or $y'$ throughout this paper, can be viewed as vectors in $\Z^n_{\ge0}$.  For example, the vertices in Figure \ref{figure1}  can  be associated with the vectors $\left[\begin{matrix}0\\0\end{matrix}\right], \left[\begin{matrix}1\\1\end{matrix}\right], \left[\begin{matrix}0\\1\end{matrix}\right], \left[\begin{matrix}2\\0\end{matrix}\right], \left[\begin{matrix}0\\2\end{matrix}\right].$  Moreover, a directed edge between two such vectors, $y \to y'$, implies a state update of the form $y'-y \in \Z^n$.  The set of  state update vectors implied by the graph is called the set of  ``reaction vectors'' for the model.   Viewing things in this manner the deficiency  for the graph provides a relation between the number of vertices, the number of connected components and the dimension of the space spanned by the reaction vectors. The formal definition of deficiency will be given in Definition \ref{def:deficiency}.


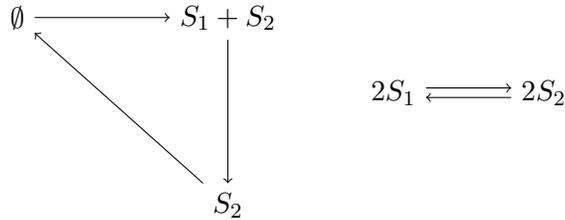
\begin{figure}\label{figure1}
\begin{center}
\begin{tikzpicture}
    \node (0)   at (4,0)  {$\emptyset$};
    \node (S1+S2)   at (6.8,0)  {$S_1+S_2$};
    \node (S2)  at (6.8,-2.5) {$S_2$};

    \path[->]
    (0) edge  (S1+S2)
    (S1+S2) edge  (S2)
    (S2) edge (0);   
    
    \node (2S1) at (9,-1) {$2S_1$};
    \node (2S2) at (11,-1) {$2S_2$};
    
    \path[->]
    ([yshift=-1.4ex]2S1.north east) edge ([yshift=-1.4ex]2S2.north west)
    ([yshift=1.4ex]2S2.south west) edge ([yshift=1.4ex]2S1.south east);

\end{tikzpicture}
\end{center}
\caption{A reaction network with two species: $S_1$ and $S_2$. The vertices are linear combinations of the species over the integers.  The directed edges are termed reactions and determine the net change in the counts of the species due to one instance of the reaction.  For example, the reaction $S_1 + S_2 \to S_2$ reduces the count of $S_1$ by one, but does not affect the count of $S_2$. 
}
\end{figure}

Given the importance of the deficiency zero property, it is natural to ask: how common is this property? There are a number of ways one can tackle this question, including a simple enumeration of all networks of a given size.  In fact, the earliest attempt to answer this question can be traced back to  work by Horn in 1973 \cite{H2}. In that paper, Horn considered all reaction networks with exactly three vertices, each of which satisfies $\sum_{i = 1}^n y_i \le 2$ where $y_i$ denotes the $i$-th component of the vector associated with a vertex $y$, but no condition on the number of species. Horn found 43 isomorphism classes of such networks, and among these, 41  have deficiency zero.

We choose a different approach by considering networks with a fixed number of species, say $n$, and then quantifying the prevalence of the  deficiency zero property via limit theorems (as $n\to \infty$) in an  Erd\H os-R\'enyi random graph framework in which there is an equal  probability, $p_n$, that there is an edge between any two vertices.
However, we are immediately confronted with a modeling problem: for any finite number of species there are an infinite number of possible graphs that can be constructed from them.      For example, with just the single species $S_1,$ possible vertices include $S_1, 2S_1, 3S_1, \dots.$  Hence, we must restrict ourselves in some manner so that for a given number of species, only a finite number of vertices are possible.  

In this paper, we restrict ourselves to study  so-called ``binary'' reaction networks, whose vertices satisfy $\sum_{i = 1}^n y_i \le 2$.   Such models are quite common in the  literature.
Our main finding is that in such a scenario $r(n)=\frac{1}{n^3}$ is a threshold function in that if $p_n \ll r(n)$, as $n\to \infty$, then the probability of deficiency zero converges to 1, whereas if $p_n \gg r(n)$, then the probability of deficiency zero converges to 0.  See Theorem \ref{cor:main1} and Theorem \ref{thm:secondmain}.
Moreover, along the way we prove that in the setting of $p_n\ll r(n)$, with high probability all the connected components of deficiency zero reaction networks will consist of pairs of vertices.  
Intriguingly, paired reaction networks can be found in certain models of autocatalytic cycles related to the study of the origin of life \cite{ecology}.

The remainder of this paper is organized as follows.  In Section \ref{sec2}, we briefly review some key terminology of reaction network theory, and provide some preliminary results related to deficiency. In Section \ref{sec3}, we set up the Erd\H os-R\'enyi random graph framework for reaction networks. In Section \ref{sec4}, we present our main results, which quantify the prevalence of deficiency zero reaction networks in our chosen framework. Finally, in Section \ref{sec5}, we end with a brief discussion.


\section{Chemical reaction networks}\label{sec2}

Here we formally introduce reaction networks and deficiency.  Moreover, we collect some preliminary results related to the deficiency of a reaction network.

\subsection{Reaction networks and key definitions}
Let $\{S_1,\dots,S_n\}$ be a set of $n$ \textit{species} undergoing a finite number of reaction types.  We denote a particular reaction by $y \to y'$, where $y$ and $y'$ are linear combinations of the species on $\{0,1,2,\dots\}$ representing the number of molecules of each species consumed and created in one instance of that reaction, respectively. The linear combinations $y$ and $y'$ are often called \textit{complexes} of the system. For a given reaction, $y\to y'$, the complex $y$ is called  the \textit{source complex} and $y'$ is called the \textit{product complex}.  A complex can be both a source complex and a product complex.   However, a complex can not be both the source and product for a single reaction nor do we include isolated complexes that are not involved in any reaction.  We may associate each complex with a vector in $\mathbb{Z}^n_{\geq 0}$, whose coordinates give the number of molecules of the corresponding species in the complex. As is common in the reaction network literature, both ways of representing complexes will be used interchangeably throughout the paper. For example,  if the system has 2 species $\{S_1,S_2\}$, the reaction $S_1+S_2\to 2S_2$ has $y=S_1+S_2$, which is associated with the vector $\left[\begin{matrix}1\\1\end{matrix}\right]$, and $y'=2S_2$, which is associated with the vector $\left[\begin{matrix}0\\2\end{matrix}\right]$.  Viewing the complexes as vectors, the \textit{reaction vector} associated to the reaction $y\to y'$ is simply $y'-y\in \Z^n$, which gives the state update of the system due to one occurrence of the reaction.

\begin{definition}\label{def:RN}
For $n \ge 0$, let $\mathcal{S} =\{S_1,...,S_n\}$, $\mathcal{C}=\cup_{y\to y'}\{y,y'\}$, and $\mathcal{R}=\cup_{y\to y'}\{y\to y'\}$ be the sets of species,  complexes, and reactions respectively. The triple $\{\mathcal{S,C,R}\}$ is called a \textit{reaction network}.  When $n = 0$, in which case $\S = \C = \Re = \emptyset$, the network is termed the \textit{empty network}. \hfill $\triangle$
\end{definition}

To each reaction network $\{\mathcal{S,C,R}\}$, there is a unique directed graph constructed in the obvious manner: the vertices of the graph are given by $\C$ and a directed edge is placed from $y$ to $y'$ if and only if $y \to y' \in \mathcal{R}$.  Figure \ref{figure1} is an example of such a graph.  
  Note that by definition the directed graph associated to a reaction network contains only vertices corresponding to elements in $\C$ involved in some reaction, i.e., the degree of all vertices is at least 1 and so isolated vertices are not present in the associated network. We denote by $\ell$ the number of  connected components of the graph.

\begin{remark}\label{remark:LCbound}
Note that since each connected component must consist of at least two vertices, we have the bound $\ell \le \frac{|\mathcal{C}|}{2}$. 
\end{remark}

 \begin{definition}
 The linear subspace generated by all reaction vectors is called the \textit{stoichiometric subspace} of the network. 
 Denote $s=\dim(span\{y'-y:y\to y'\in\Re\})$ the dimension of the stoichiometric subspace.\hfill $\triangle$
 \end{definition}
Note that $s \le n$, where $n$ is the number of species.  This fact will be used a number of times in this paper.
 
\begin{definition} 
A vertex, $y \in \Z^n_{\ge 0}$, is called \textit{binary} if $\sum_{i = 1}^n y_i = 2$. A vertex is called \textit{unary} if $\sum_{i = 1}^n y_i = 1$. The vertex  $\vec{0} \in \Z^n$ is said to be of \textit{zeroth order}. \hfill $\triangle$
\end{definition}

\begin{definition}
A reaction network $\{\mathcal{S,C,R}\}$ is called \textit{binary} if each vertex is binary, unary, or of zeroth order.\hfill $\triangle$
\end{definition}

As discussed in the introduction, we will focus on \textit{binary} reaction networks in this paper.  

The following type of network will play a key role in the current paper.

\begin{definition}\label{def:paired}
A reaction network is called \textit{paired} if each of its connected components  contains precisely two vertices. A reaction network is called \textit{$i$-paired} if it is paired and contains $i$ connected components. \hfill $\triangle$
\end{definition}

\subsection{Deficiency of a reaction network}

\begin{definition}\label{def:deficiency}
The \textit{deficiency} of a reaction network $\{\mathcal{S,C,R}\}$ is $\delta=|\mathcal{C}|-\ell-s$, where $|\mathcal{C}|$ is the number of vertices, $\ell$ is the number of connected components of the associated graph, and $s = \dim(\text{span}\{y'-y: y\to y'\in \Re\})$ is the dimension of the stoichiometric subspace of the network.

For each $j\le \ell$, we let $\C_j$ denote the collection of vertices in the $j$th connected component, $s_j$ be the corresponding dimension of the span of the reaction vectors of that component, and  define $\delta_j = |\C_j| - 1 - s_j$ to be the deficiency of that component.\hfill $\triangle$
\end{definition}

We collect a number of basic properties of deficiency in the following lemma.

\begin{lemma}\label{lemma:defproperties}
Let $n \ge 1$ and let $\{\S,\C,\Re\}$ be a reaction network with $n$ species.  
\begin{enumerate}[(a)]
    \item  $\delta$ does not depend upon the direction of the edges.
    \item  $s_j \le |\C_j| - 1,$ and so $\delta_j \ge 0.$
    \item $s \le |\C| - \ell,$ and so $\delta \ge 0.$
    \item $\delta = 0$ if and only if both the following conditions hold:
    \begin{enumerate}[(i)]
        \item $s_j = |\C_j| - 1$ for each $j\le \ell$ (equivalently, $\delta_j= 0$ for each $j\le \ell$).
        \item $\sum_{j=1}^\ell s_j = s.$
    \end{enumerate}  
    \item
    If $\delta = 0$, then
    \[
        |\C| \le 2n.
    \]
    \item 
    
    Suppose the reaction network is paired, and that $\zeta_j$ is a reaction vector from the $j$th connected component.  Then $\delta=0$ if and only if $\cup_{j=1}^\ell\{\zeta_j\} = \{\zeta_1,\dots,\zeta_\ell\}$ are linearly independent.
    
    \item (Monotonicity of deficiency.) Let $\{\widehat{\S},\widehat{\C},\widehat{\Re}\}$ and $\{\S,\C,\Re\}$  be two reaction networks with $\widehat{\Re}\setminus \Re = \{y\to y'\}$, a single reaction. Let $\hat \delta$ and $\delta$ be the deficiencies of the two networks. Then
\[
\hat\delta \geq \delta.
\]
\end{enumerate}
\end{lemma}

\begin{proof}
\begin{enumerate}[(a)]
\item This follows from the definition of deficiency.
\item This follows from the observation that a cycle within a connected component implies a dependency among the reaction vectors.
\item This follows from (b) since $\C = \cup_{j=1}^\ell \C_j$ and $s \le \sum_{j=1}^\ell s_j.$
\item This follows in a straightforward manner from $(b)$ and $(c).$
\item From the definition of deficiency $\delta = |\mathcal{C}|-\ell-s$, the fact that $s\leq n$, and $\ell\leq \frac{|\mathcal{C}|}{2}$ (from Remark \ref{remark:LCbound}), we have   
\[
\delta \geq |\mathcal{C}|- \frac{|\mathcal{C}|}{2} -n =\frac{|\mathcal{C}|}{2}-n.
\]
Since the reaction network has deficiency zero, we therefore have 
\begin{equation}\label{eq12314123}
0\geq \frac{|\mathcal{C}|}{2} -n,
\end{equation}
which implies $|\mathcal{C}|\le 2n.$ 
\item Since the reaction network is paired, we have $s_j=1$ and $|\C_j|=2$ for each $j\leq \ell$. Thus condition $(i)$ in (d) is  satisfied. Since $s_j=1$, condition $(ii)$ in (d) holds if and only if all $\zeta_j$ are linearly independent.

\item Let $\ell,s$ and $\hat{\ell},\hat{s}$ be the number of connected components and dimension of the stoichiometric subspace of $\{\S,\C,\Re\}$ and $\{\widehat{\S},\widehat{\C},\widehat{\Re}\}$, respectively. 
\begin{itemize}
    \item Case 1: $y,y'\in\C$ and $y$ and $y'$ are from the same connected component. In this case, we have $|\widehat{\C}|=|\C|$ and $\hat{\ell}=\ell$. Since $y$ and $y'$ are from the same connected component, the reaction vector $y'-y$ can be written as the linear combination of the remaining reaction vectors from its connected component. Therefore adding $y\to y'$ to $\{\S,\C,\Re\}$ does not increase the dimension of its stoichiometric subspace. Thus $\hat{s}=s$ and $\hat\delta=\delta$.
    \item Case 2: $y,y'\in\C$ and $y$ and $y'$ are from different connected components. In this case, we have $|\widehat{\C}|=|\C|$ and $\hat{\ell}=\ell-1$. Since we are adding one reaction to $\{\S,\C,\Re\}$ to obtain $\{\widehat{\S},\widehat{\C},\widehat{\Re}\}$, we add at most 1 dimension to the stoichiometric subspace of $\{\S,\C,\Re\}$. Thus $\hat{s}\leq s+1$ and 
    \[
    \hat\delta=|\widehat{\C}|-\hat{\ell}-\hat{s} \geq |\C| - (\ell-1) - (s+1) = \delta.
    \]
    \item Case 3: $y \in \C$ and $y' \notin \C$ or vice versa. In this case, we have $|\widehat{\C}|=|\C|+1$, and $\hat{\ell}=\ell$. Similar to the previous case, we must have $\hat{s}\leq s+1$, and thus
    \[
    \hat\delta=|\widehat{\C}|-\hat{\ell}-\hat{s} \geq |\C|+1 - \ell - (s+1) = \delta.
    \]
    \item Case 4: $y,y'\notin \C$. In this case, we have $|\widehat{\C}|=|\C|+2$, and $\hat{\ell}=\ell+1$. Similar to the previous cases, we still have $\hat{s}\leq s+1$ and thus
     \[
    \hat\delta=|\widehat{\C}|-\hat{\ell}-\hat{s} \geq |\C|+2- (\ell+1) - (s+1) = \delta.\qedhere
    \]
\end{itemize}
\qedhere
\end{enumerate}
\end{proof}

\begin{remark}
Lemma \ref{lemma:defproperties}(g) implies that if we remove a reaction from a reaction network with deficiency zero, then the resulting network also has deficiency zero. This means deficiency zero is a monotone decreasing property, which guarantees that a threshold function for deficiency zero exists (see \cite{threshold}).  
\end{remark}


We will illustrate the concept of deficiency via two examples.  

\begin{example}[Enzyme kinetics \cite{ACK:product}]	
Consider a reaction network with species $\{S,E,SE,P\}$ and associated graph
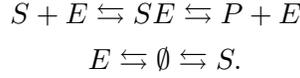
\begin{align*}
S+E &\leftrightarrows  SE \leftrightarrows P+E\\
&E \leftrightarrows \emptyset \leftrightarrows S.
\end{align*}
In this example, the reaction network has $|\mathcal{C}|=6$ vertices, there are $\ell=2$ connected components, and the dimension of the stochiometric subspace is $s=4$. Thus the deficiency is
\[
\delta = 6-2-4=0. 
\]
\hfill $\square$
\end{example}

The following example demonstrates that it is sometimes most natural to use Lemma \ref{lemma:defproperties}(f) to verify that a network has a deficiency of zero. 
\begin{example}[Binary, 3-paired network]	
Consider a reaction network with species $\{S_1, S_2, \dots, S_9\}$ and associated graph
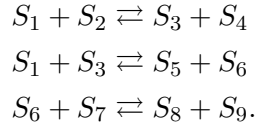
\begin{align*}
&S_1 + S_2 \rightleftarrows S_3 + S_4\\
&S_1 + S_3 \rightleftarrows S_5 + S_6\\
&S_6+S_7 \rightleftarrows S_8+S_9.
\end{align*}
This network is \textit{paired} in the sense of Definition \ref{def:paired}.  Moreover, there is linear independence among the connected components, which can  be seen easily since each connected component has a species not found in any other connected component.  Hence,   Lemma \ref{lemma:defproperties}(f) implies that the deficiency of this network is zero.
\hfill $\square$
\end{example}

\section{Erd\H os-R\'enyi model for binary reaction networks}\label{sec3}

As alluded to in the introduction, the vast majority of reaction network models found in the literature are binary.  Hence, those are the focus of the current paper.


Let the set of species be $\mathcal{S}=\{S_1,S_2,\dots,S_n\}$. We consider binary reaction networks with species in $\mathcal{S}$. The set of all possible vertices is then
\[
\C^0_n = \{\emptyset, S_i, S_i+S_j:  \text{for  $1 \le i \le n$ and $1 \le j \le n$}.\}
\] 
For a given $n$, we denote  $N_n = |\C^0_n|$, the cardinality of $\C^0_n$.  Thus, $N_n$ is  the total number of possible zeroth order, unary and binary vertices that can be generated from $n$ distinct species.  A straightforward calculation gives
\[
	N_n=1+n+n+\frac{n(n-1)}{2}= \frac{n^2+3n+2}{2},
\]
and so
\[
n \sim \sqrt{2N_n}.
\]
Here we use the notation $\sim$ in the standard way: for any two sequences of real numbers $\{a_n\}$ and $\{b_n\}$, we write $a_n\sim b_n$ if $\lim_{n\to \infty} \frac{a_n}{b_n} = c$ for some constant $c \in \R$.

We consider an Erd\H os-R\'enyi random graph $G(N_n,p_n)$, which we will simply denote $G_n$ throughout, where the set of vertices is the set 
$\C^0_n$, and the probability that there is an edge between any 2 particular vertices is $p_n$, independently of all other  edges. To each such random graph a reaction network  graph  can  be  associated  in  the  following  way
\begin{enumerate}
\item each vertex with positive degree in the random graph represents a vertex in the reaction network, and
\item each edge in the random graph represents a reaction in the reaction network  (we can assume all reactions are reversible, i.e., that $y\to y'\in \Re \implies y'\to y \in \Re$, since we do not need to worry about direction--see Lemma \ref{lemma:defproperties}(a)).
\end{enumerate}
We will denote the reaction network associated with the graph $G(N_n,p_n)$ by $R_n$. We will denote the deficiency of $R_n$ by $\delta_{R_n}.$

In order to solidify the notation, we present below the cases $n=1$ and $n=2$.  

\begin{example}[The case with $n=1$ species]
Denote the only species by $A$. The set of vertices, or equivalently the set of all possible  complexes, is $\C^0_1=\{\emptyset ,A,2A\}$. Figure \ref{figure2} shows one possible realization of the random  graph $G(N_1,p)$ when $p \in (0,1).$
The associated reaction network $R_1$ for the particular  graph shown in Figure \ref{figure2} is $\emptyset \leftrightarrows A \leftrightarrows 2A.$
\begin{figure}
\centering
\begin{tikzpicture}
\filldraw [black] (0,2) circle (2pt) node[anchor=south]{\large $\emptyset$};
\filldraw [black] (1.7,-1) circle (2pt) node[anchor=north]{\large $A$};
\filldraw [black] (-1.7,-1) circle (2pt) node[anchor=north]{\large $2A$};
\draw[-, line width=0.7mm] (0,2) -- (1.7,-1);
\draw[-, line width=0.7mm] (1.7,-1) -- (-1.7,-1);
\end{tikzpicture}
\caption{A realization of a random graph when $n=1$ and $p\in (0,1).$  The associated reaction network is $\emptyset \leftrightarrows A \leftrightarrows 2A.$}
\label{figure2}
\end{figure}
\hfill $\square$
\end{example}

\begin{example}[The case with $n=2$ species]
Denote the set of species by $\mathcal{S}=\{A,B\}$. The set of vertices is $\C^0_2=\{\emptyset,A,B,2A,2B,A+B\}$. Figure \ref{figure3} illustrates a possible  realization of the random graph $G(N_2,p)$ when $p \in (0,1)$.  The associated reaction network $R_2$ for the particular  graph shown in Figure \ref{figure2} is 
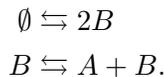
\begin{align*}
    \emptyset &\leftrightarrows 2B\\
    B &\leftrightarrows A+B.
\end{align*}
    
\begin{figure}
\centering
\begin{tikzpicture}
\filldraw [black] (0,2) circle (2pt) node[anchor=south]{\large $\emptyset$};
\filldraw [black] (0,-2) circle (2pt) node[anchor=north]{\large $A+B$};
\filldraw [black] (1.7,1) circle (2pt) node[anchor=west]{\large $B$};
\filldraw [black] (1.7,-1) circle (2pt) node[anchor=west]{\large $2B$};
\filldraw [black] (-1.7,1) circle (2pt) node[anchor=east]{\large $A$};
\filldraw [black] (-1.7,-1) circle (2pt) node[anchor=east]{\large $2A$};
\draw[-, line width=0.5mm] (0,2) -- (1.7,-1);
\draw[-, line width=0.5mm] (0,-2) -- (1.7,1);
\end{tikzpicture}
\caption{A realization of a random graph when $n=2$ and $p\in (0,1).$}
\label{figure3}
\end{figure}
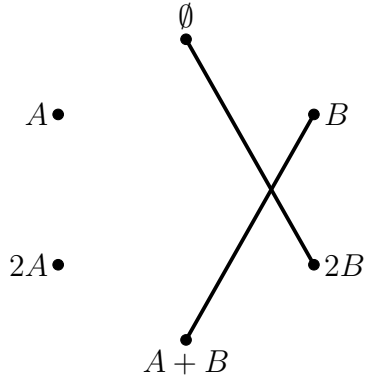
\hfill $\square$
\end{example}

\section{The threshold function for deficiency zero}\label{sec4}

The  goal of this section is to prove that $r(n) = \frac1{n^3}$ is a threshold function in that
\begin{align}\label{eq:987907}
\lim_{n\to\infty}\mathbb{P}(\delta_{R_n} = 0) = \begin{cases} 0, &\text{if } p_n \gg r(n)\\
1 & \text{if } p_n \ll r(n). \end{cases} 
\end{align}


\noindent Throughout this section, we will make use of the standard notation of $a_n \ll b_n$ or $b_n \gg a_n$ to mean $\lim_{n\to \infty} \frac{a_n}{b_n} = 0$, whenever $\{a_n\}$ and $\{b_n\}$ are sequences of non-negative real numbers.  We also remind the reader that we write $a_n \sim b_n$ to mean $\lim_{n\to \infty} \frac{a_n}{b_n} = c$ for some constant $c \in \R_{>0}$.

\subsection{The case $p_n\gg r(n)$} \label{sec:upper}


This case is relatively straightforward.  We will show that if $p_n \gg r(n)=\frac{1}{n^3}$, then with high probability we have $|\C|>2n$.  In this case,  Lemma \ref{lemma:defproperties}(e) implies the associated reaction network does not have deficiency zero.

Since $|\C|$ is the number of non-isolated vertices in $G_n$, we start with a lemma regarding the number of isolated vertices in $G_n$. 
\begin{lemma}\label{upper}
Suppose $p_n = \frac{2n+\alpha_n}{N_n(N_n-1)}$ with $\alpha_n \gg n^{1/2}$. Let $I$ be the set of isolated vertices in $G_n$, that is $I =\{ v\in \C^0_n : deg(v)=0\}$. Then we have 
\[
\lim_{n\to\infty}\mathbb{P}(|I|\geq N_n-2n)=0. 
\]
\end{lemma}
\begin{proof}
We require both $\E ( |I|)$ and $\text{Var} (|I|)$.  First, a straightforward calculation yields
\begin{align*}
\E (|I|) &= \E \left[\sum_{v\in \C^0_n} 1_{\{\deg(v)=0\}}\right] =N_n\mathbb{P}(\deg(v)=0)=N_n(1-p_n)^{N_n-1}.
\end{align*}
Turning to the variance, we have
\begin{align*}
\left| I \right|^2 = \sum_{v,w \in \C^0_n}1_{\{\deg(v)=\deg(w)=0\}} = \sum_{v\in \C^0_n}1_{\{\deg(v)=0\}}+ \sum_{v,w \in \C^0_n: v\neq w}1_{\{\deg(v)=\deg(w)=0\}}.
\end{align*}
Therefore, we  have 
\begin{align*}
\var (|I|) &= \E(|I|^2) - \left(\E(|I|)\right)^2 \\
&= \E\left[\sum_{v\in \C^0_n}1_{\{\deg(v)=0\}}+ \sum_{v,w \in \C^0_n: v\neq w}1_{\{\deg(v)=\deg(w)=0\}}\right]-  N_n^2(1-p_n)^{2N_n-2} \\
&=N_n(1-p_n)^{N_n-1} +N_n(N_n-1)(1-p_n)^{2N_n-3} - N_n^2(1-p_n)^{2N_n-2}\\
&=N_n(1-p_n)^{N_n-1}(1-(1-p_n)^{N_n-2})+N_n^2(1-p_n)^{2N_n-3}p_n\\
&\leq N_n(1-p_n)^{N_n-1}(N_n-2)p_n + N_n^2(1-p_n)^{2N_n-3}p_n\\
&\leq N_n(N_n-2)p_n+N_n^2p_n \leq 2N_n^2p_n,
\end{align*}
where the first inequality follows from Bernoulli's inequality.

We will utilize $\E(|I|)$ and $\var(|I|)$ to show that
\begin{equation}\label{eqjk12kjhio}
\lim_{n\to\infty}\P(|I|\geq N_n-2n) = 0.
\end{equation}

It suffices to prove \eqref{eqjk12kjhio} in the three cases below.
\begin{enumerate}

    \item When $\alpha_n \sim N_n$, we have $p_n\sim \frac{1}{N_n}$, and thus $p_n > \frac{c}{N_n}$ for some constant $c>0$ and $n$ large enough. Therefore
    \[
    \E(|I|) = N_n(1-p_n)^{N_n-1} \le N_n\left(1-\frac{c}{N_n}\right)^{N_n-1}\le N_n e^{-c}.
    \]
    Applying Chebyshev's inequality yields
    \[
      \P(|I|>N_n-2n) \leq \frac{\var(|I|)}{(N_n-2n-E[|I|])^2}\le\frac{2N_n^2p_n}{(N_n-2n-N_ne^{-c})^2} = \frac{2p_n}{(1-2n/N_n-e^{-c})^2}.
    \]
    Since $p_n\sim\frac{1}{N_n}$ and $N_n \sim n^2$, we  have
    \[
    \lim_{n\to\infty}\P(|I|>N_n-2n)=0.
    \]
    \item The next case is when $\alpha_n\ll N_n$, or equivalently when $p_n\ll\frac{1}{N_n}$. Using Taylor's expansion, we have
    \[
    \E(|I|)=N_n(1-p_n)^{N_n-1} \leq N_n\bigg(1-p_n(N_n-1)+p_n^2\frac{(N_n-1)(N_n-2)}{2}\bigg).
    \]
    Again, we apply Chebyshev's inequality:
\begin{align*}
\mathbb{P}(|I|\geq N_n-2n) &\leq \frac{\var(|I|)}{(N_n-2n-E[|I|])^2}\\
&\leq \frac{2N_n^2p_n}{\bigg(N_n-2n-N_n+N_n(N_n-1)p_n-\frac{N_n(N_n-1)(N_n-2)}{2}p_n^2\bigg)^2}\\
&=\frac{2N_n^2p_n}{\bigg(-2n+N_n(N_n-1)p_n-\frac{N_n(N_n-1)(N_n-2)}{2}p_n^2\bigg)^2}\\
\end{align*}
Now we plug in $p_n=\frac{2n+\alpha_n}{N_n(N_n-1)}$ and proceed:
\begin{align*}
\mathbb{P}(|I|\geq N_n-2n) &\leq  \frac{\frac{2N_n}{N_n-1}(2n+\alpha_n)}{\bigg(-2n+2n+\alpha_n-\frac{N_n-2}{2N_n(N_n-1)}(2n+\alpha_n)^2\bigg)^2}\\
&= \frac{2N_n}{N_n-1}\frac{2n+\alpha_n}{\bigg(\alpha_n-\frac{N_n-2}{2N_n(N_n-1)}(2n+\alpha_n)^2\bigg)^2}.
\end{align*}
If $\alpha_n\ll n$ or $\alpha_n\sim n$, we have 
\[
\frac{2n+\alpha_n}{\bigg(\alpha_n-\frac{N_n-2}{2N_n(N_n-1)}(2n+\alpha_n)^2\bigg)^2} \sim \frac{n}{\alpha_n^2} \to 0,
\]
as $n\to\infty$, since $\alpha_n \gg n^{1/2}$.

If $\alpha_n\gg n$, we have
\[
\frac{2n+\alpha_n}{\bigg(\alpha_n-\frac{N_n-2}{2N_n(N_n-1)}(2n+\alpha_n)^2\bigg)^2} \sim \frac{\alpha_n}{\alpha_n^2} = \frac{1}{\alpha_n} \to 0,
\]
as $n\to\infty$,  since $\alpha_n\ll N_n$. Thus, either way we must have
\[
    \lim_{n\to\infty}\P(|I|>N_n-2n)=0.
\]
  \item When $\alpha_n \gg N_n$, we can apply the same argument as Case 2, since 
  \[
\frac{2n+\alpha_n}{\bigg(\alpha_n-\frac{N_n-2}{2N_n(N_n-1)}(2n+\alpha_n)^2\bigg)^2} \ll \frac{\alpha_n}{\alpha_n^2} = \frac{1}{\alpha_n} \to 0.
\]

\end{enumerate}
In all cases above, we have $\lim_{n\to\infty}\mathbb{P}(|I|\geq N_n-2n)=0$.
\end{proof}

We are now ready to provide the first main theorem.
\begin{theorem}\label{cor:main1}
Let $G_n$ denote the Erd\H os-R\'enyi random graph with $N_n$ vertices and edge probability $p_n$, and let $R_n$ be the reaction network associated to $G_n$. When $p_n \gg r(n)=\frac1{n^3}$, we have
\[
\lim_{n\to\infty} \mathbb{P}(\delta_{R_n}=0)=0.
\]
\end{theorem}
\begin{proof}
Note that  the vertices of the reaction network $R_n$ correspond to the vertices in  $G_n$ with positive degree. Thus, letting $I$ denote the set of isolated vertices of $G_n$, Lemma \ref{lemma:defproperties}(e) implies that if the network is deficiency zero, we must have
\begin{equation}\label{eq12141231}
|I| \geq   N_n-2n .
\end{equation}
From \eqref{eq12141231}, we have
\begin{equation}\label{eq189472}
\mathbb{P}(\delta_{R_n} =0)\leq \mathbb{P}(|I|\geq N_n-2n).
\end{equation}
Since $ r(n)=\frac{1}{n^3}\sim \frac{n}{N_n^2}$ and $p_n\gg r(n)$, we have $p_n$ satisfies the condition in Lemma \ref{upper}. Hence, using Lemma \ref{upper} we  have 
\[
\lim_{n\to\infty} \mathbb{P}(\delta_{R_n}=0)=\lim_{n\to \infty} \mathbb{P}(|I| \ge N_n - 2n) = 0.\qedhere
\]

\end{proof}

\subsection{The case $p_n\ll r(n)$} \label{sec:lower}

We will show that in the case $p_n\ll r(n)$ it is enough to focus on paired reaction networks, which are introduced in Section \ref{sec2}. We first state the main theorem.

\begin{theorem}\label{thm:secondmain}
Let $G_n$ denote the Erd\H os-R\'enyi random graph with $N_n$ vertices and edge probability $p_n$, and let $R_n$ be the reaction network associated to $G_n$. When $p_n \ll r(n)=\frac1{n^3}$, we have
\[
\lim_{n\to\infty}\mathbb{P}(\delta_{R_n} =0)=1,
\]
\end{theorem}
\begin{proof}
We have
\begin{align*}
\P(\delta_{R_n}=0)&=\P(\delta_{R_n}=0, R_n \text{ is paired})+\P(\delta_{R_n}=0, R_n \text{ is not paired})\\
&\geq \P(\delta_{R_n}=0, R_n \text{ is paired}).
\end{align*}
Therefore it suffices to show
\[
\lim_{n\to\infty}\P(\delta_{R_n}=0, R_n \text{ is paired})=1.
\]
Noting that for deficiency zero models, the number of reversible reaction vectors is bounded above by $n$, we have  
\begin{align}\label{eq12i3uljio}
\mathbb{P}(\delta_{R_n}=0, R_n \text{ is paired} )& = \sum_{i=1}^n\mathbb{P}(\delta_{R_n}=0, R_n \text{ is $i$-paired} ) \nonumber\\
&=\sum_{i=1}^n\mathbb{P}(\delta_{R_n}=0| R_n \text{ is $i$-paired} )\mathbb{P}(R_n \text{ is $i$-paired}) \nonumber\\
&=\sum_{i=1}^n \mathbb{P}(\delta_{R_n}=0| R_n \text{ is $i$-paired} )\frac{N_n!}{i!2^i(N_n-2i)!}p_n^i (1-p_n)^{N_n(N_n-1)/2-i}\nonumber\\
&\geq \sum_{i=1}^n \mathbb{P}(\delta_{R_n}=0| R_n \text{ is $i$-paired} )\frac{(N_n-2i)^{2i}}{i!2^i}p_n^i (1-p_n)^{N_n(N_n-1)/2-i} 
\end{align}
where the third equality uses that the number of $i$-paired graphs is ${N_n\choose 2}{N_n-2\choose 2}\dots {N_n-2i+2\choose 2}$, with the repetition of the graphs accounted for by division  by $i!$.

Note that because $p_n \ll 1/n^3$ and $N_n \sim n^2$ we have that $N_n^2p_n \ll n$.  Now let $k_n$ satisfy $\lim_{n\to\infty}k_n=\infty$ and  $N_n^2 p_n \ll k_n\ll n$. 
Cutting off the last $n-k_n$ terms from  \eqref{eq12i3uljio}, yields
\begin{align}
\mathbb{P}(\delta_{R_n}=0, R_n \text{ is paired} )&\ge \sum_{i=1}^{k_n} \mathbb{P}(\delta_{R_n}=0| R_n \text{ is $i$-paired} )\frac{(N_n-2i)^{2i}}{i!2^i}p_n^i (1-p_n)^{N_n(N_n-1)/2-i} \nonumber\\
&\geq\sum_{i=1}^{k_n}\bigg(1-c\frac{i^4}{n^4}\bigg)\bigg(1-\frac{21i}{n}\bigg)  \frac{(N_n-2i)^{2i}}{i!2^i}p_n^i (1-p_n)^{N_n(N_n-1)/2-i} \label{technical ineq}\\
&\geq \bigg(1-c\frac{k_n^4}{n^4}\bigg)\bigg(1-\frac{21k_n}{n}\bigg)(1-p_n)^{N_n^2/2} \sum_{i=1}^{k_n}\frac{(N_n-2i)^{2i}}{i!2^i}p_n^i \nonumber\\
&\geq \bigg(1-c\frac{k_n^4}{n^4}\bigg)\bigg(1-\frac{21k_n}{n}\bigg)(1-p_n)^{N_n^2/2} \sum_{i=1}^{k_n}\frac{(N_n-2k_n)^{2i}}{i!2^i}p_n^i.\nonumber
\end{align}
where the inequality in \eqref{technical ineq} will be proven using Lemma \ref{4species}, Proposition \ref{prop:random_matrix}, and Lemma \ref{pairs} after the proof of the main theorem.  The inequalities after \eqref{technical ineq} follow by noting that $i \le k_n.$

Let $\lambda_n = \frac{(N_n-2k_n)^2p_n}{2}$, and note that $\lambda_n\ll k_n$ since we chose $N_n^2p_n\ll k_n$. Using Taylor's remainder theorem and Stirling's approximation, we have
\[
\sum_{i=1}^{k_n}\frac{\lambda_n^i}{i!} \geq e^{\lambda_n} - \frac{e^{\lambda_n}\lambda_n^{k_n+1}}{(k_n+1)!}\geq e^{\lambda_n}\bigg(1-\frac{\lambda_n^{k_n+1}}{\sqrt{2\pi}(k_n+1)^{k_n+1}e^{-k_n+1}}\bigg)=e^{\lambda_n}\bigg(1-\frac{1}{\sqrt{2\pi}}\bigg(\frac{\lambda_n e}{k_n+1}\bigg)^{k_n+1}\bigg).
\]
Thus we have
\[
\mathbb{P}(\delta_{R_n}=0, R_n \text{ is paired} )\geq  \bigg(1-c\frac{k_n^4}{n^4}\bigg)\bigg(1-\frac{21k_n}{n}\bigg)(1-p_n)^{N_n^2/2} e^{\lambda_n}\bigg(1-\frac{1}{\sqrt{2\pi}}\bigg(\frac{\lambda_n e}{k_n+1}\bigg)^{k_n+1}\bigg).
\]
Since $\lambda_n\ll k_n\ll n$, the first, second, and last terms converge to one.  Hence, it suffices to show
\[
\lim_{n\to\infty}(1-p_n)^{N_n^2/2} e^{\lambda_n} =1,
\]
or
\[
\lim_{n\to\infty}\frac{N_n^2}{2}\ln(1-p_n) + \lambda_n  =0.
\]
Since $p_n\ll 1$, we have $-p_n-p_n^2\leq \ln(1-p_n) \leq -p_n$. Thus
\[
 \frac{N_n^2}{2}\ln(1-p_n) + \lambda_n \leq -\frac{N_n^2}{2}p_n +\lambda_n =\frac{p_n}{2}((N_n-2k_n)^2-N_n^2)=\frac{p_n}{2}(-4k_nN_n + 4k_n^2).
\]
On the other hand, and using the equality above,
\[
 \frac{N_n^2}{2}\ln(1-p_n) + \lambda_n \geq -\frac{N_n^2}{2}(p_n+p_n^2) +\lambda_n = \frac{p_n}{2}(-4k_nN_n + 4k_n^2) - \frac{N_n^2p_n^2}{2}.
\]
Since $k_n\ll n$, $N_n\sim n^2$ and $p_n\ll \frac{1}{n^3}$, we have
\[
\lim_{n\to\infty}\frac{p_n}{2}(-4k_nN_n + 4k_n^2) = 0, \quad \text{and,} \quad \lim_{n\to\infty}\frac{N_n^2p_n^2}{2} =0. 
\]
Thus
\[
\lim_{n\to\infty}\frac{N_n^2}{2}\ln(1-p_n) + \lambda_n  =0,
\]
which concludes the proof of the theorem.
\end{proof}

We complete this section by providing the required technical lemmas and proposition leading to the inequality in \eqref{technical ineq}.
Recall that we only consider binary reaction networks, thus each reaction can contain at most 4 species (2 species in each vertex). 
The next  lemma shows that for our analysis later, it suffices to only consider reaction networks for which each reaction vector has exactly four non-zero components.

Note that in the construction we are using, random graphs with the same number of edges have the same probability. We use this fact heavily in the proofs of the next two lemmas, where we condition on $R_n$ being $k_n$-paired and  can therefore generate $R_n$ uniformly from the set of all $k_n$-paired graphs.

\begin{lemma}\label{4species}
Suppose that $k_n\ll n$. Let $A_n$ be the event that each reaction vector in $R_n$ has   exactly 4 non-zero components. Then we have
\[
\P(A_n|R_n \text{ is } \text{$k_n$-paired})\geq 1-\frac{21k_n}{n}
\]
\end{lemma}
\begin{proof}
Let $R_n$ be a $k_n$-paired reaction network, where $k_n\ll n$. Denote the $k_n$ reaction vectors by $\{v_n^i\}_{i=1}^{k_n} \in \mathbb{Z}^{n}$. We denote by $A_n^i$ the event that the vector $v^i_n$ has $4$ non-zero elements, thus $A_n=\cap_{i=1}^{k_n}A_n^i$.  The proof will proceed by using that
\begin{align*}
\P(A_n|R_n \text{ is } \text{$k_n$-paired}) =\prod_{j=0}^{k_n-1}\P(A_n^{j+1}|\cap_{i=1}^{j}A_n^i,R_n \text{ is } \text{$k_n$-paired}),
\end{align*}
and showing the limit of the right-hand side, as $n\to \infty$, is 1.

First, note that the total number of vertices of the form $S_k+S_m$ where $k\neq m$ is ${n \choose 2}$. Suppose we have already picked $j$ pairs of reversible reactions where each pair has 4 species. Then the number of unpicked vertices of the form $S_k+S_m$ where $k\neq m$ is ${n\choose 2}-2j$. After picking one such $S_k+S_m$ for the $j+1^{st}$ pair, we need to pick another vertex. The number of available vertices of the form $S_p+S_q$, where $p,q,m,$ and $k$ are all different is at least ${n-2\choose 2}-2j$, where the minus 2 comes from the fact that we remove the species $S_k$ and $S_m$ from the possibilities, and the $2j$ is the number of vertices we have already chosen.  

Thus for $n$ large enough, we have
\begin{align*}
\P(A_n^{j+1}|&\cap_{i=1}^{j}A_n^i,R_n \text{ is } \text{$k_n$-paired}) \\
&\geq \frac{\frac{1}{2}({n\choose 2}- 2j)({n-2\choose 2}- 2j)}{{N_n-2j \choose 2}}\tag{by considering our choices as detailed above}\\
&\ge \frac{\frac{1}{2}({n\choose 2}- 2n)({n-2\choose 2}- 2n)}{{N_n \choose 2}}\tag{since $j \le n$}\\
&=\frac{(n^2-5n)(n^2-9n+6)}{(n^2+3n+2)(n^2+3n)}\ge\frac{(n^2-5n)(n^2-9n)}{(n^2+4n)(n^2+3n)}\\
&=\frac{n^2-14n+45}{n^2+7n+12}=1-\frac{21n-33}{n^2+7n+12}\\
&\ge 1-\frac{21}{n},
\end{align*}  
and where the $1/2$ in the first term  accounts for the symmetry between the selected vertices.

Therefore, for $n$ large enough, we have
\begin{align}\label{eq12i31125}
\P(A_n|R_n \text{ is } \text{$k_n$-paired}) =\prod_{j=0}^{k_n-1}\P(A_n^{j+1}|\cap_{i=1}^{j}A_n^i,R_n \text{ is } \text{$k_n$-paired})\ge\bigg(1-\frac{21}{n}\bigg)^{k_n}\geq 1-\frac{21k_n}{n}
\end{align}
where the last inequality is due to Bernoulli's inequality. 
\end{proof}

Lemma \ref{4species} showed that if $k_n\ll n$ and $R_n$ is $k_n$-paired, then with high probability each reaction vector will have precisely 4 non-zero components.   The following proposition, stated in terms of discrete random matrices, proves that with probability approaching one, as $n \to \infty$, this set of reaction vectors will be linearly independent.

For each $n \ge 4$, let $D_n \subset \R^n$ be a set of  vectors for which (i) each vector in $D_n$ has precisely four non-zero elements, and (ii) for each  choice of four distinct indices from $\{1,\dots,n\}$ there is precisely one vector in $D_n$ with those as its non-zero components (so the size of $D_n$ is $\binom{n}{4}$).  While  the specific values of the non-zero elements do not play a role in the subsequent proposition, we note that these values are $+1$ and $-1$ in the current paper.   

\begin{proposition}\label{prop:random_matrix}
Let $k_n\ll n$ and let $\Gamma_n \in \R^{n\times k_n}$   be a  matrix whose columns are  distinct vectors chosen uniformly from $D_n$.  Let $I_n$ be the event that all column vectors of $\Gamma_n$ are linearly independent.  Then there is a constant $c>0$ for which
\[
\P(I_n) \geq 1- c\frac{k_n^4}{n^4}.
\]
\end{proposition}
\begin{proof}
We denote the $k_n$ column vectors of $\Gamma_n$ by $\{v_n^i\}_{i=1}^{k_n} \in \mathbb{R}^{n}$. We say a set of vectors is \textit{minimally dependent} if any of its proper subsets are linearly independent. For any set of indices of vectors $T\subseteq \{1,2,\dots,k_n\}$ we denote $V_n^T =\{v_n^i: i\in T\}$. By noting that
\[
I_n^c = \bigcup_{\ell = 2}^{k_n} \{\exists \text{ a minimally dependent set of size $\ell$}\},
\]
we have
\begin{align}\label{eq12938913}
\P(I_n^c) &\leq \sum_{\ell=2}^{k_n}\sum_{|T|=\ell}\P(V_n^T \text{ is minimally dependent})=\sum_{\ell=2}^{k_n}{k_n \choose l}\P(B_\ell)
\end{align}
where $B_\ell$  is the event that $V_n^T$ is minimally dependent for a particular set $T$  satisfying  $|T|=\ell$.

Now fix a set $T$ with $|T|=\ell$. Without loss of generality, let $T=\{1,2,\dots,\ell\}$. Consider a matrix $M_\ell$ whose columns are the vectors in $V_n^T$. Note that the set $V_n^T$ being minimally dependent implies that $M_\ell$ has no row with only one non-zero entry (for otherwise, the set of vectors without the column associated to that element would be linearly dependent).  This implies further that each non-zero row of $M_\ell$ has at least $2$ entries.   Since each column of $M_\ell$ has exactly $4$ non-zero entries, $M_\ell$ has exactly $4\ell$ non-zero entries. Therefore, the number of non-zero rows in $M_\ell$ must be at most $2\ell$ and the number of zero rows in $M_\ell$ must be at least $n-2\ell$. Combining all of the arguments above, we must have
\begin{align}\label{eq123jkl123}
\P(B_\ell)\leq \P(M_\ell \text{ has at least } n-2\ell \text{ zero rows}).
\end{align}

We denote the row vectors of $M_\ell$ by $\{w_n^i\}_{i=1}^{n}$. For a subset of indices of species  $R\subseteq \{1,2,\dots,n\}$ we denote $W_n^R=\{w_n^i: i\in R\}$.   We say that $W_n^R=0$ if all the vectors in the set are the zero vector. We have
\begin{align}
\begin{split}\label{eq989898}
\P(M_\ell \text{ has at least } n-2\ell \text{ zero rows}&)
\leq \sum_{|R|=n-2\ell}\P(W_n^R=0) ={n\choose n-2\ell}\P(C_\ell)
\end{split}
\end{align}
where $C_\ell$ is the event that $W_n^R=0$ for a particular $R$ satisfying $|R|=n-2\ell$. 

Now fix a set $R$ with $|R|=n-2\ell$. Without loss of generality, let $R=\{2\ell+1,\dots,n\}$. Then the event $C_\ell$ involves picking $\ell$ column vectors: $V_n^T=\{v_n^1,\dots,v_n^\ell\}$ where the last $n-2\ell$ elements of each column vector are zero. Recall that each column vector has exactly 4 non-zero elements. Suppose we have already picked $j$ such column vectors. The number of ways we can pick the $j+1$-st vector is at least $({n\choose 2}-2j)({n-2\choose 2} -2j)$ (this follows from the same argument as in the proof of Lemma \ref{4species}).  Among these, the number of ways we can pick the $j+1$-st vector whose  last $n-2\ell$  elements are zero is less than ${2\ell \choose 2} {2\ell-2 \choose 2}$.  Thus we have
\begin{align*}
 \P(C_\ell) &\leq \prod_{j=0}^{\ell-1} \frac{{2\ell \choose 2} {2\ell-2 \choose 2}}{({n\choose 2}-2j)({n-2\choose 2} -2j)}
 \leq  \prod_{j=0}^{\ell-1} \frac{{2\ell \choose 2} {2\ell-2 \choose 2}}{\frac{1}{4}{n\choose 2}{n-2\choose 2}} \leq 4 \bigg(\frac{2\ell}{n}\bigg)^4,
\end{align*}
where the 2nd inequality is due to the fact that $j\ll n$.
Plugging the above into \eqref{eq989898}, we see
\begin{align}
\begin{split}\label{eq123jo1}
\P(M_\ell \text{ has at least } n-2\ell \text{ zero rows}&) 
\leq {n\choose n-2\ell}4\bigg(\frac{2\ell}{n}\bigg)^{4\ell} \leq \frac{n^{2\ell}}{(2\ell)!} 4\bigg(\frac{2\ell}{n}\bigg)^{4\ell} \\
&\leq \frac{4n^{2\ell}}{\sqrt{2\pi}(2\ell/e)^{2\ell}}\bigg(\frac{2\ell}{n}\bigg)^{4\ell} = \frac{4}{\sqrt{2\pi}}\bigg(\frac{2\ell e}{n}\bigg)^{2\ell}.
\end{split}
\end{align}

Now combining \eqref{eq12938913}, \eqref{eq123jkl123}, and \eqref{eq123jo1}, we have
\begin{align}
\begin{split}\label{eqkjb123kj}
\P(I_n^c) &\leq \sum_{\ell=2}^{k_n} {k_n\choose \ell}\frac{4}{\sqrt{2\pi}}\bigg(\frac{2\ell e}{n}\bigg)^{2\ell}\leq\sum_{\ell=2}^{k_n}  \frac{k_n^\ell}{\ell!}\frac{4}{\sqrt{2\pi}}\bigg(\frac{2\ell e}{n}\bigg)^{2\ell}\\
&\leq\sum_{\ell=2}^{k_n}  \frac{k_n^\ell}{\sqrt{2\pi}(\ell/e)^\ell}\frac{4}{\sqrt{2\pi}}\bigg(\frac{2\ell e}{n}\bigg)^{2\ell} =\sum_{\ell=2}^{k_n}\frac{2}{\pi}\bigg(\frac{4\ell e^3 k_n}{n^2}\bigg)^\ell\\
&\leq\sum_{\ell=2}^{\infty}\frac{2}{\pi}\bigg(\frac{4 e^3 k_n^2}{n^2}\bigg)^\ell \leq c \frac{k_n^4}{n^4}.
\end{split}
\end{align}
for some constant $c>0$, since $k_n\ll n$.  Thus we have
\[
\P(I_n) \geq 1 - c\frac{k_n^4}{n^4}.\qedhere
\]
\end{proof}

We return to the setting of reaction networks with our final key lemma.

\begin{lemma}\label{pairs}
Suppose that $k_n\ll n$. Then we have
\begin{equation*}
\P(\delta_{R_n}=0|R_n \text{ is } \text{$k_n$-paired}) \geq \bigg(1-c\frac{k_n^4}{n^4}\bigg)\bigg(1-\frac{21k_n}{n}\bigg).
\end{equation*}
\end{lemma}
\begin{proof}
Let $R_n$ be a $k_n$-paired reaction network, where $k_n\ll n$. From Lemma \ref{lemma:defproperties}, $R_n$ has deficiency zero if and only if all $k_n$ reaction vectors are linearly independent. Let $I_n$ be the event that all $k_n$ reaction vectors are linearly independent.


Similar to Lemma \ref{4species}, denote by $A_n$ the event that all reactions have exactly 4 species. We have
\begin{align}\label{eqj13k2n}
 \P(\delta_{R_n}=0|R_n \text{ is } \text{$k_n$-paired})& = \P(I_n|R_n \text{ is } \text{$k_n$-paired})\nonumber \\ & \geq \P(I_n|A_n,R_n \text{ is } \text{$k_n$-paired})\P(A_n|R_n \text{ is } \text{$k_n$-paired}).
\end{align}
Utilizing Lemma \ref{4species} and Proposition \ref{prop:random_matrix}, we complete the proof of  Lemma \ref{pairs}.
\end{proof}

\section{Discussion}\label{sec5}

This work stemmed from a natural question pertaining to reaction networks: given the importance of deficiency zero in the reaction network literature, can we quantify how prevalent the condition is?   In the  Erd\H os-R\'enyi framework we have chosen here, we have provided a threshold function, $r(n) = \frac{1}{n^{3}}$, for the property in that if $p_n/r(n) \to 0$, then the probability of deficiency zero converges to 1, and if $p_n/r(n) \to \infty$, then the probability of deficiency zero converges to 0.  

We do not make the claim that the framework selected here is the only, or even the  most, biologically relevant.  Instead, having equal probabilities for each edge puts  as few assumptions on our model as possible, thereby making it a reasonable starting point for analysis.  
In fact, there are multiple avenues for future research, and we list just a few here.
\begin{itemize}
    \item One may want to study models in which some added structure is known.  For example, our assumption of equal probabilities would need to be relaxed in those contexts where different reaction types are more likely to appear in the network than others (such as when in-flows and out-flows of species are common).  This would necessitate the use of a stochastic block model framework.  We have carried out such an analysis in \cite{anderson2020deficiency}.
    
    \item In the setting of molecular biology some proteins may be more active and interact with many other proteins while some proteins may be relatively inactive and have fewer interactions. In such cases, we can study  random reaction networks under a more general random graph framework such as the Chung-Lu model, where vertices  can be assigned different weights \cite{ChungLu}. 
    
    \item Situations can arise in which some species are chemostated,  which keeps their concentrations constant. In such a case we may want to focus on the asymptotic behavior of  ``sub-networks", which consist of the species not being chemostated, instead of the whole network.    The study of sub-networks may also be useful in the multi-scale settings, where we want to focus on a subset of ``discrete'' species which are in low abundances and behave differently than those in high abundance.   \cite{ACK:ACR}.
    
    
    \item There are other meaningful topological features beside deficiency zero that we could study with our approach. Some features of interest are deficiency one (together with additional graphical features) as in \cite{ACR}, endotactic, strongly endotactic, and asyphonic as in \cite{ADE:deviation,ADE:geometric,GMS:geometric}.

\end{itemize}

 The analysis and methods developed here will, to varying degrees, be applicable to each of the situations listed above.
 


\section*{Acknowledgements}

We thank Robin Pemantle for influencing the formulation of the question posed in this paper.  We also thank an anonymous reviewer whose comments greatly influenced the final version of this paper. We gratefully acknowledge grant support from the Army Research Office via grant W911NF-18-1-0324.

\bibliographystyle{plain}
	\bibliography{bib}
	
\end{document}